\documentclass{article}
\usepackage{amsmath}
\usepackage{amsfonts}
\usepackage{amssymb}
\usepackage{amscd}
\usepackage{amsbsy}
\usepackage{amsthm}
\usepackage{graphicx}
\usepackage{psfrag}
\usepackage{color}
\usepackage{amsmath}
\usepackage{amsfonts}
\usepackage{amssymb}
\usepackage{amscd}
\usepackage{amsbsy}
\usepackage{amsthm}
\usepackage{graphicx}
\usepackage{psfrag}
\usepackage{color}
\usepackage[twoside]{geometry}
\usepackage[twoside]{geometry}
\newtheoremstyle{thmstyle}
  {6pt}
  {6pt}
  {\it}
  {}
  {\bf}
  {}
  {.5em}
  {}
\newtheoremstyle{remstyle}
  {6pt}
  {6pt}
  {\rm}
  {}
  {\bf}
  {}
  {.5em}
  {}

\theoremstyle{thmstyle}
\newtheorem{thm}{\indent Theorem}[section]

\newtheorem{prop}{\indent Proposition}[section]
\newtheorem{coro}{\indent Corollary}[section]
\newtheorem{defi}{\indent Definition}[section]

\theoremstyle{remstyle}
\newtheorem{rem}{\indent \bf Remark}[section]

\def\nd{\noindent}

\def\lu{\underline{\mu}}
\def\ou{\overline{\mu}}

\def\os{\overline{\sigma}}

\def\vp{\varphi}

\def\cp{\mathcal{P}}

\def\lu{\underline{\mu}}
\def\ou{\overline{\mu}}

\def\os{\overline{\sigma}}

\def\be{\hat{\mathbb{E}}}

\def\vp{\varphi}

\title{\Large \bf \boldmath\ \\  A note on the convergence rate of Peng's law of large numbers under sublinear expectations} 

\author{\large Mingshang Hu $^{1}$, Xiaojuan Li $^1$ and  Xinpeng Li $^{2\dagger}$} 

\date{}

\begin{document}

\maketitle

\renewcommand{\thefootnote}{\fnsymbol{footnote}}

\footnotetext{\hspace*{-5mm} \begin{tabular}{@{}r@{}p{13.4cm}@{}}
$^1$ & Zhangtai Securities Institute for Financial Studies, Shandong University,
Jinan 250100, China.\\
$^{2}$ & Research Center for Mathematics and Interdisciplinary Sciences, Shandong University,
Qingdao 266237, China.\\
$^{\dagger}$ & Corresponding author\ \ {E-mail:lixinpeng@sdu.edu.cn} \\
 $^{\ast}$ & Project supported by
  National Key R\&D Program of
China (No.2018YFA0703900) and NSF (No.11601281, No.11671231).
\end{tabular}}

\renewcommand{\thefootnote}{\arabic{footnote}}

\begin{abstract} 
This short note provides a new and simple proof of the convergence rate for Peng's law of large numbers under sublinear expectations,
which improves the corresponding results in Song \cite{song1} and Fang et al. \cite{FPSS}.
\vskip 4.5mm

\nd \begin{tabular}{@{}l@{ }p{10.1cm}} {\bf Keywords } &
Law of large numbers, rate of convergence, sublinear expectation
\end{tabular}

\nd {\bf AMS Subject  Classification } 
60F05

\end{abstract}

\baselineskip 14pt

\setlength{\parindent}{1.5em}

\setcounter{section}{0}

\section{Introduction} \label{section1}

The first law of large numbers (LLN for short) on sublinear expectation space was proved by Peng in 2007 for uncorrelated random variables on arXiv (math.PR/0702358v1), see also Peng \cite{P2019b}. The notions of independence and identical distribution (i.i.d. for short) are initiated by \cite{P2019b}, and the more general form of LLN for i.i.d. sequence is proved by Peng \cite{P2019}:

Let $\{X_i\}_{i=1}^\infty$ be an i.i.d. sequence on sublinear expectation space $(\Omega,\mathcal{H},\be)$ with $\ou=\be[X_1]$ and $\lu=-\be[-X_1]$, we further assume that
$\lim_{\lambda\rightarrow+\infty}\be[(|X_1|-\lambda)^+]=0,$ then
\begin{equation}\label{eq20}
\lim_{n\rightarrow\infty}\be\left[\vp\left(\frac{X_1+\cdots+X_n}{n}\right)\right]=\max_{\lu\leq r\leq\ou}\vp(r),
\end{equation}
where $\vp$ is continuous function satisfying linear growth condition.

Song \cite{song1} gives the following error estimation for Peng's LLN via Stein's method:
\begin{equation}\label{s1}
\sup_{|\vp|_{Lip}\leq 1}\left|\be\left[\vp\left(\frac{S_n}{n}\right)\right]-\max_{\lu\leq r\leq\ou}\vp(r)\right|\leq \frac{C}{\sqrt{n}},
\end{equation}
where $|\vp|_{Lip}\leq 1$ means that the Lipschitz constant of $\vp$ is not exceed 1 and $C$ is a constant depending only on $\be[X_1^2]$. The corresponding proof in \cite{song1} is based on the smooth approximations of nonlinear partial differential equation.

In this short note, we will provide a simple and purely probabilistic proof of (\ref{s1}). One basic tool in our proof is Chatterji's inequality in the classical probability theory (see Chatterji \cite{cha}), which says $E[|X_1+\cdots+X_n|^p]\leq 2^{2-p}\sum_{j=1}^nE[|X_j|^p]$ for martingale-difference sequence with $\max_{1\leq j\leq n}E[|X_j|^p]<\infty$, where $p\in [1,2]$. In particular, we show that the constant $C$ in (\ref{s1}) can be chosen as the upper standard deviations of $X_1$, i.e.,  $C=\os(X_1):=\inf_{\mu\in[\lu,\ou]}\be[|X_1-\mu|^2]^{\frac{1}{2}}$.

The remainder of this paper is organized as follows. Section 2 describes some basic concepts and notations of the sublinear expectation theory. The main results of this note with the simple proof are provided in Section 3.

\section{Preliminaries}

We recall some basic notions and results in the theory of sublinear
expectations. The readers may refer to \cite{HP21,P07a,P08a,pengsur,P2019} for
more details.

Let $\Omega$ be a given set and let $\mathcal{H}$ be a linear space of real
functions defined on $\Omega$ such that $c\in \mathcal{H}$ for all constants
$c$ and $|X|\in \mathcal{H}$ if $X\in \mathcal{H}$. We further suppose that if
$X_{1},\ldots,X_{n}\in \mathcal{H}$, then $\varphi(X_{1},\cdots,X_{n}%
)\in \mathcal{H}$ for each $\varphi \in C_{b.Lip}(\mathbb{R}^{n})$, where
$C_{b.Lip}(\mathbb{R}^{n})$ denotes the space of bounded and Lipschitz
functions on $\mathbb{R}^{n}$. $\mathcal{H}$ is considered as the space of
random variables. $X=(X_{1},\ldots,X_{n})$, $X_{i}\in \mathcal{H}$, is called
a $n$-dimensional random vector.

\begin{defi}
A sublinear expectation $\hat{\mathbb{E}}$ on $\mathcal{H}$ is a functional $\hat{\mathbb{{E}}}:\mathcal{H}\rightarrow \mathbb{R}$ satisfying the following
properties: for all $X,Y\in \mathcal{H}$, we have
\begin{description}
\item[(a)] Monotonicity: $\mathbb{\hat{E}}[X]\geq \mathbb{\hat{E}}[Y]$ if
$X\geq Y$.

\item[(b)] Constant preserving: $\mathbb{\hat{E}}[c]=c$ for $c\in \mathbb{R}$.

\item[(c)] Sub-additivity: $\mathbb{\hat{E}}[X+Y]\leq \mathbb{\hat{E}%
}[X]+\mathbb{\hat{E}}[Y]$.

\item[(d)] Positive homogeneity: $\mathbb{\hat{E}}[\lambda X]=\lambda
\mathbb{\hat{E}}[X]$ for $\lambda \geq0$.
\end{description}

The triple $(\Omega,\mathcal{H},\mathbb{\hat{E}})$ is called a sublinear
expectation space.
\end{defi}

Denote by $\mathcal{\hat{H}}$ the completion of $\mathcal{H}$ under the norm
$||X||:=\mathbb{\hat{E}}[|X|]$. Noting that $|\mathbb{\hat{E}}[X]-\mathbb{\hat
{E}}[Y]|\leq \mathbb{\hat{E}}[|X-Y|]$, $\mathbb{\hat{E}}[\cdot]$ can be
continuously extended to $\mathcal{\hat{H}}$. One can check that
$(\Omega,\mathcal{\hat{H}},\mathbb{\hat{E}})$ is still a sublinear expectation
space, which is called a complete sublinear expectation space. In the
following, we always suppose that $(\Omega,\mathcal{H},\mathbb{\hat{E}})$ is
complete.

\begin{defi}
Let $X$ and $Y$ be two random variables on $(\Omega,\mathcal{H},\mathbb{\hat
{E}})$. $X$ and $Y$ are called identically distributed, denoted by
$X\overset{d}{=}Y$, if for each $\varphi \in C_{b.Lip}(\mathbb{R})$,
\[
\mathbb{\hat{E}}[\varphi(X)]=\mathbb{\hat{E}}[\varphi(Y)].
\]

\end{defi}

\begin{defi}
\label{new-de1}Let $\left \{  X_{n}\right \}  _{n=1}^{\infty}$ be a sequence of
random variables on $(\Omega,\mathcal{H},\mathbb{\hat{E}})$. $X_{n}$ is said
to be independent of $\left(  X_{1},\ldots,X_{n-1}\right)  $ under
$\mathbb{\hat{E}}$, if for each $\varphi \in C_{b.Lip}(\mathbb{R}^{n})$
\[
\mathbb{\hat{E}}\left[  \varphi \left(  X_{1},\ldots,X_{n}\right)  \right]
=\mathbb{\hat{E}}\left[  \left.  \mathbb{\hat{E}}\left[  \varphi \left(
x_{1},\ldots,x_{n-1},X_{n}\right)  \right]  \right \vert _{\left(  x_{1}%
,\ldots,x_{n-1}\right)  =\left(  X_{1},\ldots,X_{n-1}\right)  }\right]  .
\]
The sequence of random variables $\left \{  X_{n}\right \}  _{n=1}^{\infty}$ is
said to be independent, if $X_{n+1}$ is independent of $\left(  X_{1}%
,\ldots,X_{n}\right)  $ for each $n\geq1$.

\end{defi}

The following representation theorem is useful in the theory of sublinear expectations.

\begin{thm}
\label{new-th1}Let $X=(X_{1},\ldots,X_{n})$ be a $n$-dimensional random
vector on $(\Omega,\mathcal{H},\mathbb{\hat{E}})$. Set%
\begin{equation}
\mathcal{P}=\{P:P\text{ is a probability measure on }(\mathbb{R}%
^{n},\mathcal{B}(\mathbb{R}^{n}))\text{, }E_{P}[\varphi]\leq \mathbb{\hat{E}%
}\left[  \varphi(X)\right]  \text{ for }\varphi \in C_{b.Lip}(\mathbb{R}%
^{n})\}. \label{e1}%
\end{equation}
Then $\mathcal{P}$ is weakly compact, and for each $\varphi \in C_{b.Lip}(\mathbb{R}^{n})$,
\begin{equation}
\mathbb{\hat{E}}\left[  \varphi
(X)\right]=\max_{P\in \mathcal{P}}E_{P}[\varphi].
\label{e3}%
\end{equation}
Moreover, if $\max_{1\leq i\leq n}\mathbb{\hat{E}}\left[  |X_{i}|^{r}\right]  <\infty$ for some $r>1$, then for each $\varphi \in C_{Lip}(\mathbb{R}^{n})$,
\begin{equation}
\mathbb{\hat{E}}\left[  \varphi(X)\right]=\max_{P\in \mathcal{P}}E_{P}[\varphi], \label{e2}%
\end{equation}
where $C_{Lip}(\mathbb{R}^{n})$ denotes the space of Lipschitz functions on
$\mathbb{R}^{n}$.
\end{thm}
The representation (\ref{e3}) is obtained by Theorem 10 in Hu and Li
\cite{HL} (see also \cite{DHP11,HP09}). Similar to Lemma 2.4.12 in Peng \cite{P2019}, it can be extend to (\ref{e2}) with higher moment comment condition.

For each given positive integer $n$, consider measurable space $(\mathbb{R}%
^{n},\mathcal{B}(\mathbb{R}^{n}))$ and define%
\begin{equation}%
\begin{array}
[c]{l}%
\tilde{X}_{i}(x)=x_{i}\text{ for }x=(x_{1},\ldots,x_{n})\in \mathbb{R}%
^{n},\text{ }1\leq i\leq n,\\
\mathcal{F}_{i}=\sigma(\tilde{X}_{1},\ldots,\tilde{X}_{i})=\{A\times
\mathbb{R}^{n-i}:\forall A\in \mathcal{B}(\mathbb{R}^{i})\},\text{ }1\leq i\leq
n,\ \mathcal{F}_{0}=\{ \emptyset,$ $\mathbb{R}^{n}\}
\end{array}
\label{e4}%
\end{equation}

The following proposition was initiated by Li \cite{Li} (see also \cite{LG,L-Z}).

\begin{prop}
\label{pr2}Let $X=(X_{1},\ldots,X_{n})$ be a $n$-dimensional random vector on
$(\Omega,\mathcal{H},\mathbb{\hat{E}})$. $\mathcal{P}$, $\tilde{X}_{i}$ and $\mathcal{F}_{i}$, $1\leq i\leq n$, are
defined in (\ref{e1}) and (\ref{e4}) respectively. If $X_{j+1}$ is independent of $\left(  X_{1}%
,\ldots,X_{j}\right)  $ for some $j\geq1$ and $\mathbb{\hat{E}}\left[
|X_{j+1}|^{1+\alpha}\right]  <\infty$ for some $\alpha>0$, then, for each $P\in \mathcal{P}$, we have%
\begin{equation}\label{ee1}
-\mathbb{\hat{E}}\left[  -X_{j+1}\right]  \leq E_{P}[\tilde{X}_{j+1}%
|\mathcal{F}_{j}]\leq \mathbb{\hat{E}}\left[  X_{j+1}\right]  ,\text{
}P-\text{a.s.}%
\end{equation}

\end{prop}

\begin{proof}

Set
$B=\{E_{P}[\tilde{X}_{j+1}|\mathcal{F}_{j}]>\mathbb{\hat{E}}\left[
X_{j+1}\right]  \} \in \mathcal{F}_{j}.$
If $P(B)>0$, then we can find a compact set $F\subset \mathbb{R}^{j}$ such that
$F\times \mathbb{R}^{n-j}\subset B$ and $P(F\times \mathbb{R}^{n-j})>0$. By
Tietze's extension theorem, there exists a sequence $\{ \varphi_{k}%
\}_{k=1}^{\infty}\subset C_{b.Lip}(\mathbb{R}^{j})$ such that $0\leq
\varphi_{k}\leq1$ and $\varphi_{k}(\tilde{X}_{1},\ldots,\tilde{X}%
_{j})\downarrow I_{F}(\tilde{X}_{1},\ldots,\tilde{X}_{j})$. Let $f(x_{j+1})=[(x_{j+1}%
-\mathbb{\hat{E}}\left[  X_{j+1}\right]  )\wedge N]\vee(-N)$, then
\[
\phi_{k,N}(x_{1},\ldots,x_{n})=\varphi_{k}(x_{1},\ldots,x_{j})f(x_{j+1}) \in
C_{b.Lip}(\mathbb{R}^{n}),
\]
for each $N\geq1$, we get%
\begin{align*}
E_{P}[\phi_{k,N}(\tilde{X}_{1},\ldots,\tilde{X}_{n})]  &  \leq \mathbb{\hat{E}%
}\left[  \phi_{k,N}(X_{1},\ldots,X_{n})\right]=\mathbb{\hat{E}}\left[ \varphi_{k}(X_{1},\ldots,X_{j})\mathbb{\hat{E}}\left[ f(X_{j+1}) \right]\right]\leq\be[f(X_{j+1})]\\
&
\leq\left \vert \mathbb{\hat{E}}\left[ f(X_{j+1})\right]  -\mathbb{\hat{E}}\left[
X_{j+1}-\mathbb{\hat{E}}\left[  X_{j+1}\right]  \right]  \right \vert\leq \frac{1}{N^\alpha}\mathbb{\hat{E}}\left[  |X_{j+1}-\mathbb{\hat{E}}\left[
X_{j+1}\right]  |^{1+\alpha}\right].
\end{align*}
Letting $N\rightarrow \infty$ first and then $k\rightarrow \infty$, we obtain%
\[
E_{P}[I_{F}(\tilde{X}_{1},\ldots,\tilde{X}_{j})(\tilde{X}_{j+1}-\mathbb{\hat
{E}}\left[  X_{j+1}\right]  )]\leq0,
\]
which contradicts to
\begin{align*}
E_{P}[I_{F}(\tilde{X}_{1},\ldots,\tilde{X}_{j})(\tilde{X}_{j+1}%
-\mathbb{\hat{E}}\left[  X_{j+1}\right]  )]&  =E_{P}[I_{F}(\tilde{X}_{1},\ldots,\tilde{X}_{j})(E_{P}[\tilde{X}%
_{j+1}|\mathcal{F}_{j}]-\mathbb{\hat{E}}\left[  X_{j+1}\right]  )]>0,
\end{align*}
since $F\times \mathbb{R}^{n-j}\subset B$ and $P(F\times \mathbb{R}%
^{n-j})>0$.

Thus $P(B)=0$, the right hand of (\ref{ee1}) holds,  so does the left hand if we consider $-X_{j+1}$.
\end{proof}

\section{Main result}

Now we give the following convergence rate of Peng's LLN.

\begin{thm}
\label{new-th3}Let $\{X_{i}\}_{i=1}^{\infty}$ be the independent random
variables on sublinear expectation space $(\Omega,\mathcal{H},\mathbb{\hat{E}%
})$ with $\mathbb{\hat{E}}[X_{i}]=\bar{\mu}$ and $-\mathbb{\hat{E}}%
[-X_{i}]=\underline{\mu}$ for $i\geq1$. Let $S_{n}=X_{1}+\cdots+X_{n}$. We further assume that there exists $\alpha\in(0,1]$ such that
\[
C_\alpha=\sup_{i\geq1}\left(
\mathbb{\hat{E}}[|X_{i}|^{1+\alpha}]\right)<\infty.
\]
Then, for each $\varphi \in C_{Lip}(\mathbb{R})$ with Lipschitz constant $L_\vp$,
we have
\[
\left \vert \mathbb{\hat{E}}\left[  \varphi \left(  \frac{S_{n}}{n}\right)
\right]  -\max_{r\in \lbrack \underline{\mu},\bar{\mu}]}\varphi(r)\right \vert
\leq L_\vp\left(\frac{4C_\alpha}{n^\alpha}\right)^{\frac{1}{1+\alpha}}.
\]

\end{thm}

\begin{proof}
For each fixed $n\geq1$, we use the notations $\mathcal{P}$, $\tilde{X}_{i}$
and $\mathcal{F}_{i}$ as in (\ref{e1}) and (\ref{e4}).


For each given $P\in \mathcal{P}$, set $\tilde{S}_{n}=\sum_{i=1}^{n}\tilde{X}_{i}${ and }$\tilde{S}_{n}^{P}%
=\sum_{i=1}^{n}E_{P}[\tilde{X}_{i}|\mathcal{F}_{i-1}]$. By Proposition
\ref{pr2}, we know $
\underline{\mu}\leq \frac{\tilde{S}_{n}^{P}}{n}\leq \bar{\mu}%
,$  $P$-a.s., which implies that
\begin{align*}
E_{P}\left[\varphi\left(\frac{\tilde{S}_{n}}{n}\right)\right]-\max_{r\in \lbrack \underline{\mu}%
,\bar{\mu}]}\varphi(r)  &  \leq E_{P}\left[\varphi\left(\frac{\tilde{S}_{n}}{n}%
\right)-\varphi\left(\frac{\tilde{S}_{n}^{P}}{n}\right)\right]\leq \frac{L_\vp}{{n}}\left(E_{P}\left[|\tilde{S}_{n}-{\tilde{S}_{n}^{P}}|^{1+\alpha}\right]\right)^{\frac{1}{1+\alpha}}
\end{align*}
Since $\{\tilde{X}_i-E_{P}[\tilde{X}_{i}|\mathcal{F}_{i-1}]\}_{i=1}^n$ is a martingale-difference sequence, by Chatterji's inequality,
\begin{equation} \label{ee3}
  \begin{split}
  E_{P}\left[|\tilde{S}_{n}-{\tilde{S}_{n}^{P}}|^{1+\alpha}\right]&\leq 2^{1-\alpha}\sum_{i=1}^nE_P\left[\left|\tilde{X}_i-E_P[\tilde{X}_i|\mathcal{F}_{i-1}]\right|^{1+\alpha}\right]   \\
  &\leq2^{1-\alpha}\sum_{i=1}^n2^\alpha\left(E_P\left[|\tilde{X}_i|^{1+\alpha}\right]+E_P\left[|E_P[\tilde{X}_i|\mathcal{F}_{i-1}]|^{1+\alpha}\right]\right)\leq 4nC_\alpha,
  \end{split}
\end{equation}

Thus, by Theorem \ref{new-th1}, we obtain
\[
\mathbb{\hat{E}}\left[  \varphi \left(  \frac{S_{n}}{n}\right)  \right]
-\max_{r\in \lbrack \underline{\mu},\bar{\mu}]}\varphi(r)=\max_{P\in \mathcal{P}%
}E_{P}\left[\varphi\left(\frac{\tilde{S}_{n}}{n}\right)\right]-\max_{r\in \lbrack \underline{\mu}%
,\bar{\mu}]}\varphi(r)\leq L_\vp\left(\frac{4C_\alpha}{n^\alpha}\right)^{\frac{1}{1+\alpha}}.
\]

On the other hand, there exists $\mu^{\ast}\in \lbrack \underline{\mu},\bar{\mu}]$ such that
$\varphi(\mu^{\ast})=\max_{r\in \lbrack \underline{\mu},\bar{\mu}]}\varphi(r)$ for fixed $\varphi \in C_{Lip}(\mathbb{R})$.
By Theorem \ref{new-th1},
we can find
$P_{i,1}, P_{i,2}\in\mathcal{P}$ such that%
$E_{P_{i,1}}[\tilde{X}_{i}]=\bar{\mu}\text{ and }E_{P_{i,2}}[\tilde{X}%
_{i}]=\underline{\mu}.
$
For $i\leq n$, if $\bar{\mu}=\underline{\mu}$, we define $P_{i}=P_{i,1}$.
Otherwise, we define
\[
P_{i}=\frac{\mu^{\ast}-\underline{\mu}}{\bar{\mu}-\underline{\mu}}%
P_{i,1}+\frac{\bar{\mu}-\mu^{\ast}}{\bar{\mu}-\underline{\mu}}P_{i,2}.
\]
One can check that $P_{i}\in \mathcal{P}$ and $E_{P_{i}}[\tilde{X}_{i}%
]=\mu^{\ast}$ for $i\leq n$, and $\tilde{X}_{1},\ldots,\tilde{X}_{n}$ are
independent under $P^{\ast}$, where $P^*$ is defined by
\[
P^{\ast}=\bigotimes \limits_{i=1}^{n}P_{i}|_{\sigma(\tilde{X}_{i})}.
\]
We can verify that $P^{\ast}\in \mathcal{P}$, $E_{P^{\ast}}[\tilde{X}_{i}%
]=E_{P^*}[\tilde{X}_i|\mathcal{F}_{i-1}]=\mu^{\ast}$.  Similar to the above proof, we have
\begin{align*}
\mathbb{\hat{E}}\left[  \varphi \left(  \frac{S_{n}}{n}\right)  \right]
-\max_{r\in \lbrack \underline{\mu},\bar{\mu}]}\varphi(r)  &  \geq E_{P^{\ast}%
}\left[  \varphi \left(  \frac{1}{n}\sum_{i=1}^{n}\tilde{X}_{i}\right)
\right]  -\varphi(\mu^{\ast})\geq-L_\vp\left(\frac{4C_\alpha}{n^\alpha}\right)^{\frac{1}{1+\alpha}}
\end{align*}
The proof is completed.
\end{proof}

In particular, if $\alpha=1$, we can give a more precise estimation.  Noting that
\begin{align}\label{ee2}
E_P[(\tilde{X_i}-E_P[\tilde{X_i}|\mathcal{F}_{i-1}])^2]&\leq\inf_{\mu\in[\lu,\ou]}E_P[(\tilde{X}_i-\mu)^2]\leq \inf_{\mu\in[\lu,\ou]}\be[({X}_i-\mu)^2],
\end{align}
we immediately have the following corollary, which generalizes the result in Song \cite{song1}.
\begin{coro}\label{c1}
Let $\{X_i\}_{i=1}^\infty$ be an i.i.d. sequence on sublinear expectation space $(\Omega,\mathcal{H},\be)$ with $\ou=\be[X_1]$, $\lu=-\be[-X_1]$ and $\be[|X_1|^2]<\infty$, Then,
\[
\sup_{|\vp|_{Lip}\leq 1}\left \vert \mathbb{\hat{E}}\left[  \varphi \left(  \frac{S_{n}}{n}\right)
\right]  -\max_{r\in \lbrack \underline{\mu},\bar{\mu}]}\varphi(r)\right \vert
\leq \frac{\os(X_1)}{\sqrt{n}},
\]
where $\os^2(X_1)=\inf_{\mu\in[\lu,\ou]}\be[|X_1-\mu|^2]$ is called the upper variance in Walley \cite{walley}.
\end{coro}
%
%
\begin{rem} Under assumptions in Corollary \ref{c1} and define the upper variance as $\os^2:=\sup_{P\in\cp}E_P[(\tilde{X_1}-E_P[\tilde{X_1}])^2]$,
Fang et al. \cite{FPSS} obtained the following LLN with rate of convergence:
\begin{equation}
\label{eee1}\be\left[d_{[\lu,\ou]}^2\left(\frac{S_n}{n}\right)\right]\leq\frac{2[\os^2+(\ou-\lu)^2]}{n},
\end{equation}
where $d_{[\lu,\ou]}(x)=\inf_{y\in[\lu,\ou]}|x-y|$. For each $P\in\cp$, by (\ref{ee3}) and (\ref{ee2}),
$$E_P\left[\left|d_{[\lu,\ou]}\left(\frac{S_n}{n}\right)\right|^2\right]\leq \frac{1}{n^2}E_{P}\left[\left|\tilde{S}_{n}-{\tilde{S}_{n}^{P}}\right|^{2}\right]\leq\frac{\os^2}{n}.$$
Thus (\ref{eee1}) can be improved by
$$\be\left[d_{[\lu,\ou]}^2(\frac{S_n}{n})\right]\leq\frac{\os^2}{n}.$$
\end{rem}



\end{document}